\documentclass[reqno]{amsart}
\usepackage[backend=bibtex,url=false]{biblatex}
\renewbibmacro{in:}{}
\usepackage{float}
\usepackage{amsmath}
\usepackage{amsthm}
\usepackage{graphicx}
\usepackage{latexsym}
\usepackage{amsfonts}
\usepackage{amssymb}
\usepackage{todonotes}
\usepackage{hyperref}
\usepackage{verbatim}
\theoremstyle{plain}
\newtheorem{theorem}{Theorem}

\newtheorem{lemma}[theorem]{Lemma}

\theoremstyle{definition}

\newtheorem*{remark*}{Remark}

\newcommand{\pr}{\mathbf P}
\renewcommand{\Pr}{\mathbf P}
\newcommand{\e}{\mathbf E}

\newcommand{\oo}{\overline}

\begin{document}

\title[First-passage times over moving boundaries]
{First-passage times over moving boundaries for asymptotically stable walks.}
\thanks{The research of A. Sakhanenko and V. Wachtel has been supported by RSF research grant No. 17-11-01173.}


\author[Denisov]{Denis Denisov}
\address{School of Mathematics, University of Manchester, UK}
\email{denis.denisov@manchester.ac.uk}

\author[Sakhanenko]{Alexander Sakhanenko}
\address{Novosibirsk State University, 
 630090 Novosibirsk, Russia}
\email{aisakh@mail.ru}

\author[Wachtel]{Vitali Wachtel}
\address{Institut f\"ur Mathematik, Universit\"at Augsburg, 86135 Augsburg, Germany}
\email{vitali.wachtel@mathematik.uni-augsburg.de}


\begin{abstract}
Let $\{S_n, n\geq1\}$ be a random walk wih independent and identically distributed
increments and let $\{g_n,n\geq1\}$ be a sequence of real numbers. Let $T_g$ denote the
first time when $S_n$ leaves $(g_n,\infty)$. Assume that the random walk is oscillating
and asymptotically stable, that is, there exists a sequence $\{c_n,n\geq1\}$ such that
$S_n/c_n$ converges to a stable law. In this paper we determine the tail behaviour of
$T_g$ for all oscillating asymptotically stable walks and all boundary sequences
satisfying $g_n=o(c_n)$. Furthermore, we prove that the rescaled random walk conditioned
to stay above the boundary up to time $n$ converges, as $n\to\infty$, towards the stable
meander.
\end{abstract}


\maketitle

\section{Introduction and main results}
Consider a one-dimensional random walk
$$
S_0=0,\ S_n=X_1+\cdots+X_n,\,n\geq1,
$$
where $X,X_1,X_2,\ldots$ are i.i.d. random variables.
For a real-valued sequence  $\{g_n\}$  let
\begin{equation}
\label{def2}
T_g:=\min\{n\geq1:S_n\leq g_n\}
\end{equation}
be the first crossing of time of the moving boundary $\{g_n\}$ by $\{S_n\}$.
The aim of this paper is to study the asymptotics of $\pr(T_g>n)$
as $n$ goes to infinity.

An important particular case of this problem is the case of
a constant boundary $g_n\equiv-x$ for some $x$. In this case
$T_g\equiv\tau_x$, where
$$
\tau_x:=\min\{n\geq1:S_n\leq -x\}.
$$
For constant boundaries the following result
(see Doney \cite{Don95}) is available: if
\begin{equation}
\label{Spitzer}
\Pr(S_n>0)\to\rho\in(0,1)
\end{equation}
then, for every fixed $x\geq0$,
\begin{equation}                           \label{iid}
\mathbf{P}(\tau_x>n)\sim V(x) n^{\rho-1}L(n),
\end{equation}
where $V(x)$ denotes the renewal function corresponding to
the weak descending ladder height process and $L(n)$ is a slowly varying function. (Here and in what
follows all unspecified limits are taken with respect to
$n\to\infty$.)

Greenwood and Novikov \cite[Theorem 1]{GN87} have shown that
if the sequence $\{g_n\}$ is decreasing and concave then
\begin{equation}
\label{GN}
\frac{\pr(T_g>n)}{\pr(\tau_0>n)}\to R_g\in(0,\infty].
\end{equation}
If, in addition, $\e |g_{\tau_0}|$ is finite, then $R_g<\infty$.
This result has been generalised by Denisov and Wachtel\cite{DW16}:
if $\{g_n\}$ decreases and $\{V(-g_n)\}$ is subadditive then \eqref{GN} holds
and $R_g$ is finite for random walks satisfying
$\e V(-g_{\tau_0})<\infty$.

If $g_n\ge 0$ is increasing, then, according to Proposition 1 in \cite{DW16},
\begin{equation}
\label{DW}
\frac{\pr(T_g>n)}{\pr(\tau_0>n)}\to L_g\in[0,1].
\end{equation}
Moreover, if $\e X=0$ and $\e X^2<\infty$ then $L_g>0$ if and only if
$\e g_{\tau_0}<\infty$. An alternative version of this result has been
obtained earlier in \cite{GN87}: it was assumed there that $\e X=0$ and
that $\e e^{-\lambda X}<\infty$ for some $\lambda>0$.

In view of \eqref{iid}, the condition $\e|g_{\tau_0}|<\infty$ is
equivalent to
$$
\sum_{n=1}^\infty |g_n|L(n)n^{\rho-2}<\infty.
$$
In particular, $\e|g_{\tau_0}|<\infty$ provided that $|g_n|=O(n^\gamma)$
with some $\gamma<1-\rho$. Since the asymptotic behaviour of the renewal
function $V$ can not be expressed in terms of $\rho$ only, it is not clear
how to use the condition $\e V(|g_{\tau_0}|)<\infty$. The trivial bound
$V(x)\le Cx$ reduces $\e V(|g_{\tau_0}|)<\infty$ to $\e|g_{\tau_0}|<\infty$.
In order to have a more accurate information on $V$ we need to impose
further restrictions on the distribution of $X$.

In the present paper we shall consider the class of asymptotically stable
random walks. Let
\begin{equation*}
\mathcal{A}:=\{0<\alpha <1;\,|\beta |<1\}\cup \{1<\alpha <2;|\beta |\leq
1\}\cup \{\alpha =1,\beta =0\}\cup \{\alpha =2,\beta =0\}
\end{equation*}%
be a subset in $\mathbb{R}^{2}.$ For $(\alpha ,\beta )\in \mathcal{A}$ and a
random variable $X$ write $X\in \mathcal{D}\left( \alpha ,\beta \right) $ if
the distribution of $X$ belongs to the domain of attraction of a stable law
with characteristic function%
\begin{equation}
G_{\alpha ,\beta }\mathbb{(}t\mathbb{)}:=\exp \left\{ -c|t|^{\,\alpha
}\left( 1-i\beta \frac{t}{|t|}\tan \frac{\pi \alpha }{2}\right) \right\},
\ c>0,  \label{std}
\end{equation}%
and, in addition, $\mathbf{E}X=0$ if this moment exists. Let
$\{c_n\}$ be a sequence of positive numbers specified by the relation%
\begin{equation}
c_n:=\inf \left\{ u\geq 0:\mu (u)\leq n^{-1}\right\},\ n\geq1 ,  \label{Defa}
\end{equation}%
where
\begin{equation*}
\mu (u):=\frac{1}{u^{2}}\int_{-u}^{u}x^{2}\mathbf{P}(X\in dx).
\end{equation*}%
It is known (see, for instance, \cite[Ch. XVII, \S 5]{F71}) that for every
$X\in \mathcal{D}(\alpha ,\beta )$ the function $\mu (u)$ is regularly
varying with index $(-\alpha )$. This implies that $c_n$ is regularly
varying with index $\alpha ^{-1}$, i.e., there exists a function $l_{1}(x)$,
slowly varying at infinity, such that
\begin{equation}
c_n=n^{1/\alpha }l_{1}(n).  \label{asyma}
\end{equation}%
In addition, the scaled sequence $\left\{ \frac{S_{n}}{c_n},\,n\geq 1\right\} $
converges in distribution to the stable law given by (\ref{std}). In this
case we say that $S_n$ is an \emph{asymptotically stable random walk}.
For every $X\in \mathcal{D}\left( \alpha ,\beta \right)$
there is an explicit formula for $\rho$,
\begin{equation}
\displaystyle\rho =\left\{
\begin{array}{ll}
\frac{1}{2},\ \alpha =1, &  \\
\frac{1}{2}+\frac{1}{\pi \alpha }\arctan \left( \beta \tan \frac{\pi \alpha
}{2}\right) ,\text{ otherwise}. &
\end{array}%
\right.  \label{ro}
\end{equation}
If $X\in\mathcal{D}(\alpha,\beta)$ then the function $V(x)$ is regularly varying
with index $\alpha(1-\rho)$. Moreover, according to Lemma~13 in \cite{VW09},
\begin{equation}
\lim_{n\to\infty} V(c_n)\pr(\tau_0>n)=:A\in(0,\infty).
\label{V-tau}
\end{equation}

By Corollary 1 in \cite{DW16}, if $S_n$ is asymptotically stable then
the finiteness of $\e V(|g_{\tau_0}|)$ is equivalent to
$$
\sum_{n=1}^\infty\frac{V(|g_n|)}{nV(c_n)}<\infty.
$$
Using the fact that the function $V(x)$ is regularly varying of index
$\alpha(1-\rho)$, we see that $\e V(|g_{\tau_0}|)$ is finite if
$|g_n|=O(c_n/\log^a n)$ with some $a>1/\alpha(1-\rho)$.

If $\{g_n\}$ is decreasing but $V(-g_n)$ is not subadditive then we can
not apply Theorem 1 from \cite{DW16}. But it is shown in Theorem 2 in the
same paper that \eqref{GN} with finite $R_g$ remains valid for boundaries
satisfying
\begin{equation}
\label{Int.Test}
\sum_{n=1}^\infty\frac{V(|g_n|)}{nV(c_n/\log n)}<\infty.
\end{equation}
Moreover, it is proven in \cite{DW16} that if $\{g_n\}$ increases and
satisfies \eqref{Int.Test} then the constant $L_g$ in \eqref{DW} is
strictly positive. We note also that \eqref{Int.Test} is fulfilled if,
for example, $g_n=O(c_n/\log^{1+a}n)$ with some $a>1/\alpha(1-\rho)$.
A logarithmic version of this result has been shown by Aurzada and
Kramm~\cite{AK16}. More precisely, they have proven that
$$
\pr(T_g>n)=n^{\rho-1+o(1)}
$$
for any boundary satisfying $g_n=O(n^\gamma)$ with some $\gamma<1/\alpha$.

In the present paper we are going to derive the asymptotics of $\pr(T_g>n)$
for all boundaries $g_n=o(c_n)$. Since $c_n$ is the scaling sequence for
the random walk $S_n$, it is natural to expect that the behaviour of
$\pr(T_g>n)$ is quite similar to the behaviour of $\pr(\tau_0>n)$. The
following result confirms this conjecture.
\begin{theorem}
\label{T1}
Assume that $X\in \mathcal{D}\left( \alpha ,\beta \right)$.
If $g_n=o(c_n)$ and $\pr(T_g>n)>0$ for all $n\geq1$ then
\begin{equation}
\label{T1.1}
\frac{\pr(T_g>n)}{\pr(\tau_0>n)}\sim U_g(n),
\end{equation}
where $U_g$ is a positive slowly varying function with values
$$
0<U_g(n)=\e[V(S_n-g_n);T_g>n],\quad n\geq1.
$$
\end{theorem}

If $\e X=0$ and $\e X^2<\infty$ then \eqref{T1.1} is a special case of
Theorem 2 from our previous paper \cite{DSW16}, where random walks with
independent but not necessarily identical distributed increments have been
considered.

Theorem~\ref{T1} states that the tail of $T_g$ is regularly varying tail
with index $\rho-1$ for any boundary $g_n=o(c_n)$. We now turn to the
question, for which boundaries the sequences $\pr(T_g>n)$ and $\pr(\tau_0>n)$
are asymptotically equivalent. In other words, we want to find conditions
which guarantee that $U_g(n)$ is bounded away from $0$ and from $\infty$.
\begin{theorem}
\label{T2}
Assume that $X\in \mathcal{D}\left( \alpha ,\beta \right)$ and that,
as $x\to\infty$,
\begin{equation}
\label{lrt}
V(x+1)-V(x)=O\left(\frac{V(x)}{x}\right).
\end{equation}
\begin{itemize}
 \item [(a)]
If
\begin{equation}
\label{T2.1}
\sum_{n=1}^\infty\frac{\max_{k\le n}|g_k|}{nc_n}<\infty
\end{equation}
then there exist positive constants $U_*$ and $U^*$ such that
\begin{equation}
\label{T2.2}
U_*\le U_g(n)\le U^*\quad\text{for all }n\geq1.
\end{equation}
\item[(b)]
Moreover, if the sequence $\{g_n\}$ is monotone and \eqref{T2.1} holds then
\begin{equation}
\label{T2.3}
\lim_{n\to\infty}U_g(n)=:U_g(\infty)\in(0,\infty).
\end{equation}
\end{itemize}
\end{theorem}

Mogulskii and Pecherskii \cite{MP79} have shown that if the boundary sequence
satisfies the condition $g_{n+k}\le g_n+g_k$, then there exists a sequence
of events $\{E_n\}$ such that
\begin{equation}
\label{WH0}
E_n\subseteq\{S_n>g_n\}\quad\text{for every}\quad n\ge1
\end{equation}
and
\begin{equation}
\label{WH1}
\sum_{n=0}^\infty z^n\mathbf{P}(T_g>n)=
\exp\left\{\sum_{n=1}^\infty\frac{z^n}{n}\mathbf{P}(E_n)\right\}.
\end{equation}
This relation is a generalisation of the classical factorisation identity for
the stopping time $\tau_0$. Unfortunately, the events $E_n$ have very complicated
structure in the case of moving boundaries and there is no hope to derive the
tail asymptotics for $T_g$ from \eqref{WH1}. But \eqref{WH0} allows one to obtain
upper bounds for $\pr(T_g>n)$. It has been shown in Remark 2 in \cite{DW16} that
$$
\pr(T_g>n)\le q_n
$$
with $q_n$ defined by
\begin{equation}
\label{WH2}
\sum_{n=0}^\infty z^nq_n=\left(\sum_{n=0}^\infty z^n\mathbf{P}(\tau_0>n)\right)
\exp\left\{\sum_{n=1}^\infty\frac{z^n}{n}\Delta_n\right\},
\end{equation}
where $\Delta_n:=\pr(S_n>g_n)-\pr(S_n>0)$. Using the standard estimate for the
concentration function of $S_n$, one gets
$$
|\Delta_n|\le C\frac{|g_n|+1}{c_n}
$$
From this bound and \eqref{WH2} we infer that if $\frac{|g_n|}{nc_n}$ is summable
then
$$
\pr(T_g>n)\le q_n\le C\pr(\tau_0>n).
$$
It is worth mentioning that the condition \eqref{T2.1} is quite close to the summability
of the sequence $\left\{\frac{|g_n|}{nc_n}\right\}$.

If the boundary sequence is strictly positive, $g_n\to\infty$ and $g_n=o(c_n)$, then,
by the local limit theorem for $S_n$,
$$
\Delta_n\sim -f_{\alpha,\beta}(0)\frac{g_n}{c_n},
$$
where $f_{\alpha,\beta}(x)$ is the density function of the stable distribution given by
\eqref{std}. If we additionally assume that $\frac{g_n}{nc_n}$ is not summable, then,
by \eqref{WH2},
$$
\pr(T_g>n)=o(\pr(\tau_0>n)).
$$
This indicates that the condition \eqref{T2.1} is very close to the optimal one, and it
cannot be relaxed in the case of monotone increasing boundaries.

We now turn to the conditional limit theorem. Define the rescaled process
\begin{equation}
\label{s-def}
s_n(t)=\frac{S_{[nt]}}{c_n},\quad t\in[0,1].
\end{equation}
It has been shown by Doney \cite{Doney85} that if $X\in\mathcal{D}(\alpha,\beta)$ then, for every fixed
$x$, $s_n$ conditioned on $\{\tau_x>n\}$ converges weakly on $D[0,1]$ towards a process
$M_{\alpha,\beta}$. This limiting process is usually called the {\it stable L\'evy meander}. Our next
result shows that this convergence remain valid for all moving boundaries satisfying $g_n=o(c_n)$.
\begin{theorem}
\label{T3}
Assume that the conditions of Theorem~\ref{T1} hold. Then the distribution of $s_n$ conditioned on
$\{T_g>n\}$ converges weakly on $D[0,1]$ towards $M_{\alpha,\beta}$.
\end{theorem}

For random walks with zero mean and finite variance we have convergence towards the Brownian meander.
In \cite{DSW16} we have proven that this convergence holds even for random walks with non-identically
distributed increments satisfying the classical Lindeberg condition. But for random walks with
infinite variance the statement of Theorem~\ref{T3} is new.

The conditional limit theorem allows one to complement Theorem~\ref{T2} by the following statement:
if $g_n=o(c_n)$ is monotone decreasing and $|g_n|/nc_n$ is not summable, then
\begin{equation}
\label{U-infinity}
\lim_{n\to\infty}U_g(n)=\infty.
\end{equation}
(We shall prove \eqref{U-infinity} at the end of the paper.)

Recall that we have shown after Theorem~\ref{T2} that if $g_n$ is increasing and $g_n/nc_n$ is not
summable then $\lim_{n\to\infty}U_g(n)=0$. This implies that the conditions on the boundary in
Theorem~\ref{T2}(b) are optimal. As a result we have determined the asymptotic behaviour of $U_g$ for
all asymptotically stable walks satisfying \eqref{lrt}, which is a bit weaker than the strong renewal
theorem for ladder heights. It is well-known from the renewal theory that the strong renewal theorem
and \eqref{lrt} hold for all walks satisfying $\alpha(1-\rho)<1/2$. But if $\alpha(1-\rho)\ge 1/2$ then
\eqref{lrt} may fail, see Example 4 in \cite{W12}. We refer to a recent paper by Caravenna and
Doney\cite{CD16} for necessary and sufficient conditions for the strong renewal theorem.

Our approach to moving boundaries is based on the following universality idea. The condition
$g_n=o(c_n)$ means that the boundary reduces to the constant zero boundary after the rescaling of the
random walk by $c_n$. Therefore, it is natural to expect that the asymptotic behaviour of $\pr(T_g>n)$
will be simiar to that of $\pr(\tau_0>n)$. This is an adaption of the universality methodology
suggested in our recent paper \cite{DSW16}, where the first-passage problems for random walks belonging
to the domain of attaraction of the Brownian motion have been considered. It is worth mentioning that
in the present paper we use a different type of universality: we fix the distribution of the random
walk and look for a possible widest class of boundary functions with the same type of the tail behaviour
for the corresponding first-pasage time.
\section{Some results from the fluctuation theory}
In this section we collect some known facts about first-passage problems with constant boundaries.
We start with the following result on exit times.
\begin{lemma}
\label{lem:const_boundary}
Let $S_n$ be an asymptotically stable random walk. Then, for every
$\delta_n\downarrow0$ there exists $\varepsilon_n\downarrow 0$ such that
\begin{equation}\label{lem1.eq1}
\sup_{x\in[0,\delta_nc_n]}\left|\frac{\pr(\tau_x>n)}{V(x)\pr(\tau_0>n)}-1\right|\le \varepsilon_n.
\end{equation}
In  addition, the following estimate is valid
for all $x\ge 0$,
\begin{equation}\label{lem1.eq2}
\pr(\tau_x>n)\le  C_0V(\min\{x, c_n\})\pr(\tau_0>n).
\end{equation}
\end{lemma}
The first statement \eqref{lem1.eq1} is Corollary~3 in \cite{D12}, and \eqref{lem1.eq2} is proven in
Lemma~2.1 in \cite{AGKV}.

Let $\tau^+$ denote the first ascending ladder epoch, that is,
$$
\tau^+:=\min\{n\geq1:S_n>0\}.
$$
Let $H(x)$ denote the renewal function of strict ascending ladder epochs.
Then, similar to \eqref{V-tau}, one has
\begin{equation}
\label{H-tau}
\lim_{n\to\infty} H(c_n)\pr(\tau^+>n)=:A^+\in(0,\infty).
\end{equation}
Define also $\tau^+_x:=\min\{n\geq1: S_n>x\}$. Then, similar to \eqref{lem1.eq2},
\begin{equation}
\label{tau+}
\pr(\tau^+_x>n)\le C_0 H(\min\{x,c_n\})\pr(\tau^+>n),\quad x\ge0.
\end{equation}

Combining \eqref{V-tau} and \eqref{H-tau}, and using the well-known relation
$$
\pr(\tau^+>n)\pr(\tau_0>n)\sim n^{-1},
$$
we conclude that
\begin{equation}
\label{VH}
\lim_{n\to\infty}\frac{V(c_n)H(c_n)}{n}\in(0,\infty).
\end{equation}

\begin{lemma}
\label{lem:meander}
Let $f$ be a continuous functional on $D[0,1]$. Then, for every $\delta_n\to0$,
$$
\sup_{x\le \delta_n c_n}\Big|\e[f(s_n)|\tau_x>n]-\e f(M_{\alpha,\beta})\Big|\to0.
$$
\end{lemma}
\begin{proof}
Let $x_n$ be a sequence satisfying $x_n\le \delta_n c_n$. Caravenna and Chaumont
have shown in \cite{CC08} that the Doob transform of $s_n$ converges to a stable
process conditioned to stay positive at all times. Performing the inverse change
of measure one can easily obtain the convergence
$$
\e[f(s_n)|\tau_{x_n}>n]\to\e f(M_{\alpha,\beta}).
$$
The desired uniformity follows from the standard contradiction argument.
\end{proof}

\section{Proof of Theorem~\ref{T1}}
\subsection{Preliminary estimates}
Define
$$
G_n:=\max_{k\leq n}|g_k|,\quad Z_n:=S_n-g_n
$$
and
$$
Q_{k,n}(y):=\pr\left(y+\min_{k\le j\le n}(Z_j-Z_k)>0\right).
$$
\begin{lemma}
\label{lem:Q-bound}
Fix some sequence $\delta_n\downarrow0$ such that $\delta_nc_n$ increases.
Then, for all $y\ge0,$
\begin{align}
\label{Q-bound.1}
\nonumber
&\max_{k\le n/2}\left|\frac{Q_{k,n}(y)}{\pr(\tau_0>n-k)}-V(y)\right|\\
&\hspace{1cm}\le
\oo{\varepsilon}_{n}V(y)+2(1+C_0+\oo{\varepsilon}_1)V(G_n)+
2C_0V(y)\mathbb{I}\{y>\oo{\delta}_nc_n-2G_n\},
\end{align}
where
$$
\oo{\varepsilon}_n:=\max_{k\in[n/2,n]}\varepsilon_k,\quad
\oo{\delta}_n:=\frac{\min_{k\in[n/2,n]}\delta_k c_k}{c_n}
$$
and $\varepsilon_n$ is taken from \eqref{lem1.eq1}.
\end{lemma}
\begin{proof}
It is immediate from the definition of $Q_{k,n}$ that
\begin{eqnarray*}
\pr\left(y-2G_n+\min_{j\leq n-k}S_k>0\right)
\le Q_{k,n}(y)\le
\pr\left(y+2G_n+\min_{j\leq n-k}S_k>0\right)
\end{eqnarray*}
If $y+2G_n\le \oo{\delta}_nc_n$ then $y+2G_n\le \delta_{n-k}c_{n-k}$ for all $k\le n/2$.
Therefore, by \eqref{lem1.eq1},
$$
\pr\left(y+2G_n+\min_{j\leq n-k}S_k>0\right)
\leq(1+\oo{\varepsilon}_n)V(y+2G_n)\pr(\tau_0>n-k)
$$
for every $y\le \oo{\delta}_nc_n-2G_n$. Using now the subadditivity of $V$, we obtain
\begin{align*}
&\frac{\pr\left(y+2G_n+\min_{j\leq n-k}S_k>0\right)}{\pr(\tau_0>n-k)}\\
&\hspace{1cm}\leq(1+\oo{\varepsilon}_n)V(y)+2(1+\oo{\varepsilon}_n)V(G_n),\quad  y\le \oo{\delta}_nc_n-2G_n.
\end{align*}
If $y> \oo{\delta}_nc_n-2G_n$ then, using \eqref{lem1.eq2} and the subadditivity of $V$, we have
$$
\frac{\pr\left(y+2G_n+\min_{j\leq n-k}S_k>0\right)}{\pr(\tau_0>n-k)}
\leq C_0V(y)+2C_0V(G_n),\quad y> \oo{\delta}_nc_n-2G_n.
$$
As a result we have
\begin{align}
\label{Q1}\nonumber
&\frac{Q_{k,n}(y)}{\pr(\tau_0>n-k)}\\
&\hspace{1cm}\le (1+\oo{\varepsilon}_n)V(y)+2(1+C_0+\oo{\varepsilon}_n)V(G_n)
+C_0V(y)\mathbb{I}\{y>\oo{\delta}_nc_n-2G_n\}.
\end{align}
If $y\le \oo{\delta}_nc_n-2G_n$ then it follows from \eqref{lem1.eq1} that
$$
\pr\left(y-2G_n+\min_{j\leq n-k}S_k>0\right)
\geq(1-\oo{\varepsilon}_n)V(y-2G_n)\pr(\tau_0>n-k).
$$
Therefore, due to the subadditivity of $V$,
\begin{align*}
\frac{Q_{k,n}(y)}{\pr(\tau_0>n-k)}&\geq
\frac{\pr\left(y-2G_n+\min_{j\leq n-k}S_k>0\right)}{\pr(\tau_0>n-k)}\\
&\geq(1-\oo{\varepsilon}_n)V(y)-2V(G_n)-V(y)\mathbb{I}\{y>\oo{\delta}_nc_n-2G_n\}.
\end{align*}
Combining this with \eqref{Q1}, we obtain \eqref{Q-bound.1}.
\end{proof}
Define
$$
Z_n^*:=V(Z_n)\mathbb{I}\{T_g>n\}.
$$
\begin{lemma}
\label{lem:Z-bound}
For every stopping time $\nu$,
$$
\left|\e Z^*_{\nu\wedge n}-\e Z^*_n\right|\leq 2V(G_n)\pr(T_g>\nu\wedge n), \quad n\geq1.
$$
\end{lemma}
\begin{proof}
By the Markov property at time $\nu\wedge n$,
\begin{align*}
\e Z^*_n&=\e[V(S_n-g_n);T_g>n]\\
&=\sum_{k=1}^n\int_0^\infty\pr(Z_k\in dz;T_g>k,\nu\wedge n=k)  \\
&\hspace{3cm}\times\e\left[V(z+Z_n-Z_k);z+\min_{k\le j\le n}Z_j-Z_k>0\right].
\end{align*}
Then, we have the following estimates from above
\begin{align*}
  \e Z^*_n&\le\sum_{k=1}^n\int_0^\infty\pr(Z_k\in dz;T_g>k,\nu\wedge n=k)\\
&\hspace{3cm}\times\e\left[V(z+2G_n+S_{n-k});z+2G_n+\min_{j\le n-k}S_j>0\right]
\end{align*}
and below
\begin{align*}
\e Z^*_n&=\e[V(S_n-g_n);T_g>n]\\
&\ge\sum_{k=1}^n\int_0^\infty\pr(Z_k\in dz;T_g>k, \nu\wedge n=k)\\
&\hspace{3cm}\times\e\left[V(z-2G_n+S_{n-k});z-2G_n+\min_{j\le n-k}S_j>0\right].
\end{align*}
Then, using the harmonicity and the subadditivity of $V(x)$, we get
\begin{align*}
\e Z^*_n&\le \e[V(Z_{\nu\wedge n}+2G_n);T_g>\nu\wedge n]\\
&\le \e Z^*_{\nu\wedge n}+2V(G_n)\pr(T_g>\nu\wedge n).
\end{align*}
and
\begin{align*}
\e Z^*_n&\ge \e[V(Z_{\nu\wedge n}-2G_n);T_g>\nu\wedge n]\\
&\ge \e Z^*_{\nu\wedge n}-2V(G_n)\pr(T_g>\nu\wedge n).
\end{align*}
Thus, the proof is complete.
\end{proof}

Define the stopping times
\begin{equation}
\label{def:nu}
\nu(h):=\min\{k\geq1: Z_k\geq h\}
\quad\text{and}\quad
\nu_n:=\nu(c_n)\wedge n.
\end{equation}

\begin{lemma}
\label{lem:Ubound}
There exist constants $C_1$ and $C_2$ such that
\begin{equation}
\label{UB1}
\frac{\pr(T_g>n)}{\pr(\tau_0>n)}\le C_1\e Z^*_n
\end{equation}
and
\begin{equation}
\label{UB2}
\frac{\pr(T_g>\nu_n)}{\pr(\tau_0>n)}\le C_2\e Z^*_n.
\end{equation}
for all $n\geq1$.
\end{lemma}
\begin{proof}
According to Lemma 24 in \cite{DSW16},
$$
\pr(S_n\ge x|T_g>n)\ge \pr(S_n\geq x),\quad x\in\mathbb{R}.
$$
This implies that
\begin{equation}
\label{UB3}
\frac{\e Z^*_n}{\pr(T_g>n)}=\e[V(Z_n)|T_g>n]\ge \e V(Z_n).
\end{equation}
Since $S_n$ is asymptotically stable and $V(x)$ is regularly varying of index $\alpha(1-\rho)$,
$$
\e V(Z_n)=\e V(S_n-g_n)\sim V(c_n)\e [Y^{\alpha(1-\rho)};Y>0],
$$
where $Y$ is distributed according to the stable law from \eqref{std}.

Combining this with \eqref{UB3}, we obtain
$$
\liminf_{n\to\infty}\frac{\e Z^*_n}{V(c_n)\pr(T_g>n)}\ge \e [Y^{\alpha(1-\rho)};Y>0]>0.
$$
Using now \eqref{V-tau}, we get \eqref{UB1}.

In order to prove \eqref{UB2} we note that
\begin{align*}
\pr(T_g>\nu_n)&=\pr(T_g>\nu_n, Z_{\nu_n}<c_n)+\pr(T_g>\nu_n, Z_{\nu_n}\ge c_n)\\
&\le \pr(T_g>n)+\pr(Z^*_{\nu_n}\geq V(c_n)).
\end{align*}
Applying \eqref{UB1} to the first summand and the Markov inequality to the second summand,
we obtain
\begin{equation}
\label{UB4}
\pr(T_g>\nu_n)\le C_1\e Z^*_n\pr(\tau_0>n)+\frac{\e Z^*_{\nu_n}}{V(c_n)}.
\end{equation}
By Lemma~\ref{lem:Z-bound},
$$
\frac{\e Z^*_{\nu_n}}{V(c_n)}\le \frac{\e Z^*_{n}}{V(c_n)}+\frac{2V(G_n)}{V(c_n)}\pr(T_g>\nu_n).
$$
Substituting this into \eqref{UB4}, we have
$$
\pr(T_g>\nu_n)\le C_1\e Z^*_n\pr(\tau_0>n)+\frac{\e Z^*_{n}}{V(c_n)}+\frac{2V(G_n)}{V(c_n)}\pr(T_g>\nu_n).
$$
Since $G_n=o(c_n)$, $2V(G_n)/V(c_n)<1/2$ for all $n$ sufficiently large. For such values of $n$ we  have
$$
\pr(T_g>\nu_n)\le 2C_1\e Z^*_n\pr(\tau_0>n)+2\frac{\e Z^*_{n}}{V(c_n)},
$$
and \eqref{UB2} follows now from \eqref{V-tau}.
\end{proof}
\begin{lemma}
\label{lem:Z-exp}
Sequences $\e Z^*_n$ and $\e Z^*_{\nu_n}$ are slowly varying and, moreover,
$$
\e Z^*_n\sim\e Z^*_{\nu_n}.
$$
\end{lemma}
\begin{proof}
Taking $\nu\equiv k<n$ in Lemma~\ref{lem:Z-bound} and using \eqref{UB1}, we obtain
\begin{align*}
\left|\e Z^*_k-\e Z^*_n\right|
\le 2V(G_n)\pr(T_g>k)\le 2C_1 V(G_n)\e Z^*_k\pr(\tau_0>k).
\end{align*}
Therefore,
\begin{align*}
\max_{k\in[m,n]}\left|\frac{\e Z^*_n}{\e Z^*_k}-1\right|\leq 2C_1V(G_n)\pr(\tau_0>m).
\end{align*}
It follows from the assumption $G_n=o(c_n)$ and \eqref{V-tau} that $V(G_n)=o(1/\pr(\tau_0>n))$.
Recalling that $\pr(\tau_0>n)$ is regularly varying, we infer that $V(G_n)=o(1/\pr(\tau_0>m(n)))$
if $\frac{m(n)}{n}\to0$ sufficiently slow.
Thus,
\begin{align*}
\max_{k\in[m(n),n]}\left|\frac{\e Z^*_k}{\e Z^*_n}-1\right|\to 0
\end{align*}
provided that $\frac{m(n)}{n}$ is bounded from below or goes to zero sufficiently slow.
In particular, the sequence $\e Z^*_n$ is slowly varying.

Taking $\nu=\nu_n$ in Lemma~\ref{lem:Z-bound} and using \eqref{UB2}, we have
\begin{align*}
\left|\e Z^*_{\nu_n}-\e Z^*_n\right|
\le 2V(G_n)\pr(T_g>\nu_n)\le 2C_2 V(G_n)\e Z^*_n\pr(\tau_0>k)=o(\e Z^*_n).
\end{align*}
In other words, $\e Z^*_{\nu_n}\sim\e Z^*_n$. Thus, the proof is finished.
\end{proof}
\begin{lemma}
\label{lem:tail_exp}
For every sequence $A_n$ satisfying $A_n\gg c_n$ we have
$$
\e[Z^*_{\nu_n}; Z_{\nu_n}>A_n]=o\left(\e Z^*_n\right).
$$
\end{lemma}
\begin{proof}
Since $V$ is increasing and subadditive, for all $n$ sufficiently large,
\begin{align*}
&\e[Z^*_{\nu_n}; Z_{\nu_n}>A_n]\\
&\hspace{1cm}=\sum_{j=1}^n\int_{g_{j-1}}^{c_n}\pr(S_{j-1}\in dy,T_g>j-1)
              \e[V(y-g_j+X_1);y-g_j+X_1>A_n]\\
&\hspace{1cm}\le\sum_{j=1}^n\pr(T_g>j-1)\e[V(c_n+2G_n+X_1);c_n+2G_n+X_1>A_n]\\
&\hspace{1cm}\le \sum_{j=1}^n\pr(T_g>j-1)
\left(\e\left[V(X_1);X_1>\frac{A_n}{2}\right]+3V(c_n)\pr\left(X_1>\frac{A_n}{2}\right)\right)
\end{align*}
Combining \eqref{UB1}, Lemma~\ref{lem:Z-exp} and the fact that $\pr(\tau_0>j)$
is regularly varying of index $\rho-1\in(-1,0)$, we get
$$
\sum_{j=1}^n\pr(T_g>j-1)\leq 1+C_1\sum_{j=1}^{n-1} \e Z^*_j\pr(\tau_0>j)\leq
C n\e Z^*_n\pr(\tau_0>n).
$$
Therefore,
\begin{align}
\label{tail.1}
\nonumber
&\frac{\e[Z^*_{\nu_n}; Z_{\nu_n}>A_n]}{\e Z^*_n}\\
&\hspace{1cm}\le C n\pr(\tau_0>n)
\left(\e\left[V(X_1);X_1>\frac{A_n}{2}\right]+3V(c_n)\pr\left(X_1>\frac{A_n}{2}\right)\right).
\end{align}
The assumption $A_n\gg c_n$ implies that $\pr(X_1>A_n)=o(n^{-1})$. Consequently,
\begin{equation}
\label{tail.2}
V(c_n)\pr\left(X_1>\frac{A_n}{2}\right)=o\left(\frac{1}{n\pr(\tau_0>n)}\right).
\end{equation}
Furthermore,
\begin{align*}
\e\left[V(X_1);X_1>\frac{A_n}{2}\right]
=\int_{A_n/2}^\infty V(x)\pr(X_1\in dx)
\le \int_{A_n/2}^\infty \frac{V(x)}{x^2}\theta(dx),
\end{align*}
where $\theta(dx):=x^2\pr(|X_1|\in dx)$. If $S_n$ is asymptotically stable then
$\Theta(x):=\theta((0,x))$ is regularly varying of index $2-\alpha$. Since
$V(x)/x^2$ is regularly varying of index $\alpha(1-\rho)-2$, we infer  that
$$
\e\left[V(X_1);X_1>\frac{A_n}{2}\right]
\le C\frac{V(A_n)}{A_n^2}\Theta(A_n)
=o\left(\frac{V(c_n)}{c_n^2}\Theta(c_n)\right),
$$
where the last step follows from the fact that $\frac{V(x)}{x^2}\Theta(x)$
is regularly varying of index $-\alpha\rho<0$. By the definition of $c_n$,
$c_n^{-2}\Theta(c_n)\sim n^{-1}$. Using \eqref{V-tau} once again, we get
\begin{equation}
\label{tail.3}
\e\left[V(X_1);X_1>\frac{A_n}{2}\right]=o\left(\frac{1}{n\pr(\tau_0>n)}\right).
\end{equation}
By combining \eqref{tail.1}--\eqref{tail.3}  we complete the proof.
\end{proof}
\subsection{Proof of Theorem~\ref{T1}}
Let $\{m(n)\}$ be a sequence of natural numbers such that $m(n)\to\infty$ and $m(n)=o(n)$.
By the Markov property,
$$
\pr(T_g>n)=\e [Q_{\nu_{m(n)},n}(Z_{\nu_{m(n)}});T_g>\nu_{m(n)}].
$$
Applying Lemma~\ref{lem:Q-bound} and noting that $\pr(\tau_0>n-k)\sim\pr(\tau_0>n)$ uniformly in
$k\le m(n)$, we get
\begin{align}
\nonumber
\label{T1.2}
&\frac{\pr(T_g>n)}{\pr(\tau_0>n)}=(1+o(1))\e Z^*_{\nu_{m(n)}}
+O\left(V(G_n)\pr(T_g>\nu_{m(n)})\right)\\
&\hspace{4cm}
+O\left(\e[Z^*_{\nu_{m(n)}};Z_{\nu_{m(n)}}>\oo{\delta}_nc_n-G_n]\right).
\end{align}
By \eqref{UB2}, $\pr(T_g>\nu_{m(n)})\le C_2\e Z^*_{m(n)}\pr(\tau_0>m(n))$.
From this estimate and from the fact that $\pr(\tau_0>n)V(G_n)\to0$ we infer that,
for every sequence $\{m(n)\}$ such that $m(n)/n\to0$ sufficiently slow,
\begin{equation}
\label{T1.3}
V(G_n)\pr(T_g>\nu_{m(n)})=o(\e Z^*_{m(n)}).
\end{equation}
For every sequence $m(n)=o(n)$ we can choose $\{\delta_n\}$ satisfying
$\oo{\delta}_nc_n\gg G_n$ and $\oo{\delta}_nc_n\gg c_{m(n)}$. Then by Lemma~\ref{lem:tail_exp},
$$
\e[Z^*_{\nu_{m(n)}};Z_{\nu_{m(n)}}>\oo{\delta}_nc_n-G_n]=o(\e Z^*_{m(n)}).
$$
Plugging this and \eqref{T1.3} into \eqref{T1.2}, we obtain
$$
\frac{\pr(T_g>n)}{\pr(\tau_0>n)}=(1+o(1))\e Z^*_{\nu_{m(n)}}+o(\e Z^*_{m(n)}).
$$
According to Lemma~\ref{lem:Z-exp},
\begin{equation}
\label{m(n)-prop}
\e Z^*_{\nu_{m(n)}}\sim \e Z^*_{m(n)}\sim \e Z^*_{n}
\end{equation}
provided that $m(n)/n\to0$ sufficiently slow. Consequently,
$$
\frac{\pr(T_g>n)}{\pr(\tau_0>n)}\sim \e Z^*_{n}.
$$
Thus, the proof is complete.
\section{Proof of Theorem~\ref{T2}}
\subsection{Technical preparations}
\begin{lemma}
\label{lem:moments}
For any sequence $\{r_n\}$ satisfying $r_n=o(c_n)$ we have
\begin{equation*}
\e[V(S_n+r_n);T_g>n]\sim \e Z_n^*.
\end{equation*}
\end{lemma}
\begin{proof}
By the subadditivity of $V(x)$,
$$
|V(x+y)-V(x)|\le V(|y|),\quad x,y\in\mathbb{R}.
$$
Therefore,
\begin{align*}
&|\e[V(S_n+r_n);T_g>n]-\e Z_n^*|\\
&\hspace{2cm}=|\e[V(S_n+r_n);T_g>n]-\e[V(S_n-g_n);T_g>n]|\\
&\hspace{2cm}\le V(|r_n+g_n|)\pr(T_g>n)
\end{align*}
According to Theorem~\ref{T1}, $\pr(T_g>n)\sim \e Z_n^*\pr(\tau_0>n)$.
Therefore,
\begin{align*}
\frac{\e[V(S_n+r_n);T_g>n]}{\e Z^*_n}-1=O\big(V(|r_n+g_n|)\pr(\tau_0>n)\big).
\end{align*}
Recalling that $|r_n+g_n|=o(c_n)$ and taking into account \eqref{V-tau}, we conclude
that $V(|r_n+g_n|)\pr(\tau_0>n)$ converges to zero. This completes the proof.
\end{proof}

\begin{lemma}
\label{lem:local}
Under the conditions of Theorem~\ref{T1} we have
$$
\pr(S_n\in(x,x+1],T_g>n)=O\left(\frac{H(\min\{x+G_n,c_n\})}{nc_n}\e Z^*_n\right)
$$
uniformly in $x$.
\end{lemma}
\begin{proof}
Set $m=[n/2]$. By the Markov property at time $m$,
\begin{align*}
&\pr(S_n\in(x,x+1],T_g>n)\\
&\hspace{1cm}\le
\int_{g_m}^\infty\pr(S_m\in dy,T_g>m)\pr(S_{n-m}\in(x-y,x-y+1],\tau_{y+G_n}>n-m).
\end{align*}
Define $X^*_k=-X_{n-m+1-k}$, $S^*_k=X^*_1+X^*_2+\ldots+X^*_k$ for $k=1,2,\ldots,n-m$.
Define also $\tau^*_y:=\min\{k\ge 1: S^*_k<-y\}$. Then
\begin{align*}
&\pr(S_{n-m}\in(x-y,x-y+1],\tau_{y+G_n}>n-m)\\
&\hspace{1cm}\le
\pr(S^*_{n-m}\in[y-x-1,y-x),\tau^*_{x+1+G_n}>n-m).
\end{align*}
Since $S_k^*$ is also asymptotically stable, one has the following standard
bound for the concentration function:
$$
\sup_x\pr(S_n^*\in(x,x+1])\le\frac{C}{c_n}.
$$
Using this bound, we infer that
\begin{align*}
&\pr(S^*_n\in(x,x+1],\tau_y^*>n)\\
&\hspace{1cm}\le\int_{-\infty}^\infty\pr(S^*_{n/2}\in(x,x+1],\tau_y^*>n/2)\pr(S_{n/2}^*\in(x-y,x-y+1])\\
&\hspace{1cm}\le\frac{C}{c_{n/2}}\pr(\tau^*_y>n/2).
\end{align*}
Therefore,
\begin{align*}
&\pr(S_{n-m}\in(x-y,x-y+1],\tau_{y+G_n}>n-m)\\
&\hspace{1cm}=
O\left(\frac{\pr(\tau^*_{x+1+G_n}>(n-m)/2)}{c_{(n-m)/2}}\right)
=O\left(\frac{\pr(\tau^*_{x+1+G_n}>n)}{c_{n}}\right).
\end{align*}
It is obvious that $\pr(\tau^*_{x+1+G_n}>n)=\pr(\tau^+_{x+1+G_n}>n)$.
Then, taking into account \eqref{tau+} and \eqref{H-tau}, we conclude that
$$
\pr(S_n\in(x,x+1],T_g>n)=O\left(\frac{H(\min\{c_n,x+G_n\})}{c_nH(c_n)}\pr(T_g>n)\right).
$$
Recalling that $\pr(T_g>n)=O(\e Z^*_n/V(c_n))$ and using \eqref{VH},
we obtain the desired bound.
\end{proof}
\begin{lemma}
\label{lem:V-diff}
Assume that the conditions of Theorem~\ref{T1} are valid.
Assume, in addition, that \eqref{lrt} holds.
Then
$$
\e[V(S_n+G_{2n})-V(S_n+G_n);T_g>n]=O\left(\frac{G_{2n}}{c_n}\e Z^*_n\right)
$$
and
$$
\e[V(S_n-G_{n})-V(S_n-G_{2n});T_g>n]=O\left(\frac{G_{2n}}{c_n}\e Z^*_n\right).
$$
\end{lemma}
\begin{proof}
We first note that the subadditivity of $V$ implies the bound
\begin{align*}
&\e[V(S_n+G_{2n})-V(S_n+G_n);S_n<G_{2n},T_g>n]\\
&\hspace{1cm}\leq V(G_{2n})\pr(S_n<G_{2n},T_g>n).
\end{align*}
Applying Lemma~\ref{lem:local}, we then get
\begin{align*}
\e[V(S_n+G_{2n})-V(S_n+G_n);S_n<G_{2n},T_g>n]\\
=O\left(G_{2n}\frac{H(G_{2n})V(G_{2n})}{nc_n}\e Z^*_n\right).
\end{align*}
Recalling that $G_n=o(c_n)$ and using \eqref{VH}, we infer that
$$
H(G_{2n})V(G_{2n})=o(n).
$$
As a result,
\begin{equation}
\label{V-diff.1}
\e[V(S_n+G_{2n})-V(S_n+G_n);S_n<G_{2n},T_g>n]=o\left(\frac{G_{2n}}{c_n}\e Z^*_n\right).
\end{equation}
Furthermore, it follows from \eqref{lrt} that, uniformly for $x\in(G_{2n},c_n)$,
\begin{align*}
\e[V(S_n+G_{2n})-V(S_n+G_n);S_n\in(x,x+1],T_g>n]\\
=O\left(G_{2n}\frac{V(x)}{x}\pr(S_n\in(x,x+1],T_g>n)\right).
\end{align*}
Applying now Lemma~\ref{lem:local}, we conclude that
\begin{align*}
\e[V(S_n+G_{2n})-V(S_n+G_n);S_n\in(x,x+1],T_g>n]\\
=O\left(G_{2n}\frac{V(x)H(x)}{xnc_n}\e Z^*_n\right).
\end{align*}
Therefore,
\begin{align*}
\e[V(S_n+G_{2n})-V(S_n+G_n);S_n\in(G_{2n},c_n],T_g>n]\\
=O\left(\frac{G_{2n}}{nc_n}\e Z^*_n\sum_{k=[G_{2n}]}^{[c_n]+1}\frac{V(k)H(k)}{k}\right).
\end{align*}
Recalling that $V(x)H(x)$ is regularly varying with index $\alpha$ and taking into account
\eqref{VH}, we arrive at
\begin{align}
\label{V-diff.2}
\e[V(S_n+G_{2n})-V(S_n+G_n);S_n\in(G_{2n},c_n],T_g>n]
=O\left(\frac{G_{2n}}{c_n}\e Z^*_n\right).
\end{align}
Using \eqref{lrt} once again and noting that the function $\frac{V(x)}{x}$ is
eventually non-increasing, we get
\begin{align*}
\e[V(S_n+G_{2n})-V(S_n+G_n);S_n>c_n,T_g>n]
=O\left(G_{2n}\frac{V(c_n)}{c_n}\pr(T_g>n)\right).
\end{align*}
By Theorem~\ref{T1} and \eqref{V-tau}, $V(c_n)\pr(T_g>n)\sim\e Z^*_n$. Consequently,
\begin{align*}
\e[V(S_n+G_{2n})-V(S_n+G_n);S_n>c_n,T_g>n]=O\left(\frac{G_{2n}}{c_n}\e Z^*_n\right).
\end{align*}
Combining this with \eqref{V-diff.1} and \eqref{V-diff.2}, we complete
the proof of the first estimate. The second one can be derived by using the same arguments.
For this reason we omit its proof.
\end{proof}
\subsection{Proof of Theorem~\ref{T2}(a)}
For every $m\in(n,2n]$ we have
\begin{align*}
&\e[V(S_m+G_m);T_g>m]\\
&\hspace{0.3cm}=\int_{-G_n}^\infty\pr(S_n\in dx;T_g>n)\e[V(x+S_{m-n}+G_m);\min_{k\leq n-m}(x+S_k-g_{n+k})>0]\\
&\hspace{0.3cm}\le \int_{-G_n}^\infty\pr(S_n\in dx;T_g>n)\e[V(x+S_{m-n}+G_{2n});\tau_{x+G_{2n}}>n-m].
\end{align*}
Recalling that $V(y+S_k)\mathbb{I}\{\tau_y>k\}$ is martingale, we obtain
\begin{align*}
&\max_{m\in(n,2n]}\e[V(S_m+G_m);T_g>m]\\
&\hspace{1cm}\le \e[V(S_n+G_{2n});T_g>n]\\
&\hspace{1cm}=\e[V(S_n+G_{n});T_g>n]+\e[V(S_n+G_{2n})-V(S_n+G_n);T_g>n].
\end{align*}
Applying the first estimate from Lemma~\ref{lem:V-diff} and noting that
$$
\e Z^*_n=\e[V(S_n-g_n);T_g>n]\le \e[V(S_n+G_n);T_g>n],
$$
we infer that, for some constant $B$ and all $n\geq 1$,
\begin{align*}
\max_{m\in(n,2n]}\e[V(S_m+G_m);T_g>m]
\le \e[V(S_n+G_{n});T_g>n]\left(1+B\frac{G_{2n}}{c_n}\right).
\end{align*}
Thus, for every $\ell\ge1$,
\begin{align*}
\max_{n\le 2^\ell}\e[V(S_m+G_m);T_g>m]
\le \e[V(S_1+G_1);T_g>1]
\prod_{j=0}^{\ell-1}\left(1+B\frac{G_{2^{j+1}}}{c_{2^{j}}}\right).
\end{align*}
It is obvious that \eqref{T2.1} implies that
$$
\sum_{j=1}^\infty\frac{G_{2^{j+1}}}{c_{2^{j}}}<\infty.
$$
Therefore,
\begin{equation*}
\sup_{n\geq1}\e[V(S_n+G_{n});T_g>n]<\infty.
\end{equation*}
Recalling that $U_g(n)=\e Z^*_n$ is bounded from above by $\e[V(S_n+G_{n});T_g>n]$,
we get the upper bound in \eqref{T2.2}.

The proof of the lower bound in \eqref{T2.1} is very similar to the proof of the upper
bound. We first note that
$$
\e Z^*_n\ge \e[V(S_n-G_n);T_g>n].
$$
Furthermore, for every $m\in(n,2n]$,
\begin{align*}
&\e[V(S_m-G_m);T_g>m]\\
&\hspace{0.3cm}\ge \int_{G_{2n}}^\infty\pr(S_n\in dx;T_g>n)\e[V(x+S_{m-n}-G_{2n});\tau_{x-G_{2n}}>n-m]\\
&\hspace{0.3cm}=\e[V(S_n-G_{2n});T_g>n]\\
&\hspace{0.3cm}=\e[V(S_n-G_{n});T_g>n]-\e[V(S_n-G_n)-V(S_n-G_{2n});T_g>n].
\end{align*}
Using the second estimate from Lemma~\ref{lem:V-diff} and recalling that,
by Lemma~\ref{lem:moments},
$\e Z^*_n\sim\e[V(S_n-G_n);T_g>n]$, we arrive at the inequality
\begin{align*}
\min_{m\in(n,2n]}\e[V(S_m-G_m);T_g>m]
\ge \e[V(S_n-G_{n});T_g>n]\left(1-B\frac{G_{2n}}{c_n}\right).
\end{align*}
Choosing $n_0$ so that $B\frac{G_{2n}}{c_n}<\frac{1}{2}$ for all $n>n_0$ we then get
\begin{align*}
&\min_{n\le n_02^\ell}\e[V(S_n-G_n);T_g>n]\\
&\hspace{0.5cm}\ge \min_{n\le n_0}\e[V(S_n-G_n);T_g>n]
\prod_{j=0}^{\ell-1}\left(1-B\frac{G_{n_02^{j+1}}}{c_{n_02^{j}}}\right).
\end{align*}
From this bound and \eqref{T2.2} we obtain the desired lower bound.
\subsection{Proof of Theorem~\ref{T2}(b)}
If $g_n$ increases, then, according to Lemma 4 in \cite{DW16}, the sequence
$V(S_n-g_n)\mathbb{I}\{T_g>n\}$ is a supermartingale. In particular, the
sequence $\e Z^*_n$ decreases and has finite limit. The positivity of the
limit follows from \eqref{T2.2}.

If $g_n$ decreases, then $V(S_n-g_n)\mathbb{I}\{T_g>n\}$ is a submartingale,
see Lemma 1 in \cite{DW16}. This implies that the limit of $\e Z^*_n$ is
positive. Its finiteness follows from \eqref{T2.2}.

\section{Functional convergence}
\subsection{Proof of the conditional limit theorem}
Fix some sequence $m(n)=o(n)$ such that \eqref{m(n)-prop} holds. Let $\delta_n$ satisfy the condition
$$
G_n\ll \oo{\delta}_n^2c_n\ll c_{m(n)}\ll \oo{\delta}_nc_n.
$$
By the Markov property and \eqref{lem1.eq2},
\begin{align*}
&\pr(T_g>n, Z_{\nu_{m(n)}}>\oo{\delta}_n c_n)\\
&\hspace{0.5cm}=\int_{\oo{\delta}_n c_n}^\infty \pr(Z_{\nu_{m(n)}}\in dz, T_g>\nu_{m(n)})
\pr\left(z+\min_{\nu_{m(n)}\le j\le n}(Z_j-Z_{\nu_{m(n)}})>0\right)\\
&\hspace{0.5cm}\le\int_{\oo{\delta}_n c_n}^\infty \pr(Z_{\nu_{m(n)}}\in dz, T_g>\nu_{m(n)})
\pr\left(z+2G_n+\min_{j\le n-m(n)}S_j>0\right)\\
&\hspace{0.5cm}\le C_0\pr(\tau_0>n-m(n))\e[V(Z_{\nu_{m(n)}}+2G_n);T_g>\nu_{m(n)},Z_{\nu_{m(n)}}>\oo{\delta}_n c_n].
\end{align*}
Since $G_n\ll \oo{\delta}_nc_n$ and $m(n)=O(n)$, we have
$V(Z_{\nu_{m(n)}}+2G_n)=O(V(Z_{\nu_{m(n)}}))$ uniformly on the evernt
$\{Z_{\nu_{m(n)}}>\oo{\delta}_n c_n\}$. Consequently,
$$
\pr(T_g>n, Z_{\nu_{m(n)}}>\oo{\delta}_n c_n)=
O\left(\pr(\tau_0>n)\e[Z^*_{\nu_{m(n)}};Z_{\nu_{m(n)}}>\oo{\delta}_n c_n]\right).
$$
Now, in view of Lemma~\ref{lem:Z-exp} and \eqref{m(n)-prop},
\begin{equation}
\label{T1.4}
\pr(T_g>n, Z_{\nu_{m(n)}}>\oo{\delta}_n c_n)=
o\left(\pr(\tau_0>n)\e Z^*_n\right).
\end{equation}
Using the Markov property and \eqref{lem1.eq2} once again, we obtain
\begin{align*}
&\pr(T_g>n, Z_{\nu_{m(n)}}<\oo{\delta}^2_n c_n)\\
&\hspace{0.5cm}=\int_0^{\oo{\delta}^2_n c_n} \pr(Z_{\nu_{m(n)}}\in dz, T_g>\nu_{m(n)})
\pr\left(z+\min_{\nu_{m(n)}\le j\le n}(Z_j-Z_{\nu_{m(n)}})>0\right)\\
&\hspace{0.5cm}\le\int_0^{\oo{\delta}^2_n c_n} \pr(Z_{\nu_{m(n)}}\in dz, T_g>\nu_{m(n)})
\pr\left(z+2G_n+\min_{j\le n-m(n)}S_j>0\right)\\
&\hspace{0.5cm}\le C_0 V(\oo{\delta}^2_n c_n) \pr(\tau_0>n-m(n))\pr(T_g>\nu_{m(n)}).
\end{align*}
Then, according to \eqref{UB2} and \eqref{m(n)-prop},
$$
\pr(T_g>n, Z_{\nu_{m(n)}}<\oo{\delta}^2_n c_n)=
O\left(V(\oo{\delta}^2_n c_n)\pr(\tau_0>n)\e Z^*_{n}\pr(\tau_0>m(n))\right).
$$
Using the relation $\pr(\tau_0>m(n))\sim C/V(c_{m(n)})$ and the assumption
$c_{m(n)}\gg\oo{\delta}^2_n c_n$, we get
$$
V(\oo{\delta}^2_n c_n)\pr(\tau_0>m(n))\to0.
$$
Therefore,
\begin{equation}
\label{T1.5}
\pr(T_g>n, Z_{\nu_{m(n)}}<\oo{\delta}^2_n c_n)=o\left(\pr(\tau_0>n)\e Z^*_n\right).
\end{equation}

Let $f$ be a uniformly continuous and bounded functional on the space $D[0,1]$.
Without loss of generality, we may assume that $0\le f\le 1$. It follows then
from \eqref{T1.4}, \eqref{T1.5} and Theorem~\ref{T1} that
\begin{align}
\label{T1.6}
\e\left[f(s_n);T_g>n\right]=
\e\left[f(s_n);Z_{\nu_{m(n)}}\in[\oo{\delta}^2_n c_n,\oo{\delta}_n c_n],T_g>n\right]
+o(\pr(T_g>n)).
\end{align}
For every $k\geq0$ and every $y\in\mathbb{R}$ define a functional $f(k,y;\cdot)$ by the following
relation:
$$
f(k,y;h):=f\left(y+\left(h(t)-h\left(\frac{k}{n}\right)\right)
\mathbb{I}\left\{t\geq \frac{k}{n}\right\}\right),\quad h\in D[0,1].
$$
It follows from the definition of $\nu_{m(n)}$ that
\begin{align*}
\frac{\max_{k\leq\nu_{m(n)}}|S_k-S_{\nu_{m(n)}}|}{c_n}
&\leq \frac{\max_{k\leq\nu_{m(n)}}|Z_k-Z_{\nu_{m(n)}}|}{c_n}+\frac{2G_n}{c_n}\\
&\leq \frac{Z_{\nu_{m(n)}}}{c_n}+\frac{3G_n}{c_n}
\leq \oo{\delta}_n+\frac{3G_n}{c_n}
\end{align*}
on the event $\{Z_{\nu_{m(n)}}\le \oo{\delta}_nc_n,T_g>\nu_{m(n)}\}$.
From this bound and the uniform continuity of the functional $f$ we infer that
$$
f(s_n)-f\left(\nu_{m(n)},\frac{S_{\nu_{m(n)}}}{c_n};s_n\right)=o(1)
\quad\text{on the event }\{Z_{\nu_{m(n)}}\le \oo{\delta}_nc_n,T_g>\nu_{m(n)}\}.
$$
Combining this estimate with \eqref{T1.6}, we obtain
\begin{align}
\label{T1.7}
\nonumber
&\e\left[f(s_n);T_g>n\right]\\
&=
\e\left[f\left(\nu_{m(n)},\frac{S_{\nu_{m(n)}}}{c_n};s_n\right);Z_{\nu_{m(n)}}\in[\oo{\delta}^2_n c_n,\oo{\delta}_n c_n],T_g>n\right]
+o(\pr(T_g>n)).
\end{align}
By the Markov property at $\nu_{m(n)}$,
\begin{align*}
&\mathbf{E}\left[f\left(\nu_{m(n)},\frac{S_{\nu_{m(n)}}}{c_n},s_n\right);
T_g>n,Z_{\nu_{m(n)}}\in[\oo{\delta}^2_n c_n,\oo{\delta}_n c_n]\right]\\
&\hspace{1cm}=\sum_{k=1}^{m(n)}\int_{\oo{\delta}^2_n c_n}^{\oo{\delta}_n c_n}\mathbf{P}(Z_{k}\in dy,\nu_{m(n)}=k, T_g>k)\\
&\hspace{3cm}\times\mathbf{E}\left[f\left(k,\frac{y+g_k}{c_n};s_n\right);y+\min_{j\in[k,n]}(Z_j-Z_k)>0\right].
\end{align*}
We now note that it suffices to show that, uniformly in $y\in[\oo{\delta}^2_n c_n,\oo{\delta}_n c_n]$
and $k\leq m(n)$,
\begin{align}
\label{T1.8}
\nonumber
&\mathbf{E}\left[f\left(k,\frac{y+g_k}{c_n},s_n\right);y+\min_{j\in[k,n]}(Z_j-Z_k)>0\right]\\
&\hspace{3cm}=(\mathbf{E}f(M_{\alpha,\beta})+o(1))V(y)\pr(\tau_0>n).
\end{align}
Indeed, this relation implies that
\begin{align*}
&\mathbf{E}\left[f\left(\nu_{m(n)},\frac{S_{\nu_{m(n)}}}{c_n},s_n\right);
T_g>n,Z_{\nu_{m(n)}}\in[\oo{\delta}^2_n c_n,\oo{\delta}_n c_n]\right]\\
&\hspace{0.2cm}=(\mathbf{E}f(M_{\alpha,\beta})+o(1))\pr(\tau_0>n)
\mathbf{E}[V(Z_{\nu_m(n)});T_g>\nu_{m(n)},Z_{\nu_{m(n)}}\in[\oo{\delta}^2_n c_n,\oo{\delta}_n c_n]].
\end{align*}
It follows from the assumption $\oo{\delta}^2_n c_n\ll c_{m(n)}$ and the definition
of $\nu_{m(n)}$ that
\begin{align*}
&\mathbf{E}[V(Z_{\nu_m(n)});T_g>\nu_{m(n)},Z_{\nu_{m(n)}}<\oo{\delta}^2_n c_n]\\
&\hspace{1cm}=\mathbf{E}[V(Z_{\nu_m(n)});T_g>m(n),Z_{m(n)}<\oo{\delta}^2_n c_n]\\
&\hspace{2cm}\le V(\oo{\delta}^2_n c_n)\pr(T_g>m(n)).
\end{align*}
Applying now Theorem~\ref{T1} and recalling that $\e Z^*_{m(n)}\sim \e Z^*_n$, we get
\begin{align*}
&\mathbf{E}[V(Z_{\nu_m(n)});T_g>\nu_{m(n)},Z_{\nu_{m(n)}}<\oo{\delta}^2_n c_n]\\
&\hspace{1cm}=O(V(\oo{\delta}^2_n c_n)\e Z^*_n\pr(\tau_0>m(n))).
\end{align*}
Using now \eqref{V-tau} and the assumption $\oo{\delta}^2_n c_n\ll c_{m(n)}$, we coclude that
\begin{align*}
\mathbf{E}[V(Z_{\nu_m(n)});T_g>\nu_{m(n)},Z_{\nu_{m(n)}}<\oo{\delta}^2_n c_n]=o(\e Z^*_n).
\end{align*}
We know that $c_{m(n)}\ll \oo{\delta}_n c_n$. Then, by Lemma~\ref{lem:tail_exp},
\begin{align*}
\mathbf{E}[V(Z_{\nu_m(n)});T_g>\nu_{m(n)},Z_{\nu_{m(n)}}>\oo{\delta}_n c_n]=o(\e Z^*_n).
\end{align*}
From the two relations we infer that
\begin{align*}
\mathbf{E}[V(Z_{\nu_m(n)});T_g>\nu_{m(n)},Z_{\nu_{m(n)}}\in[\oo{\delta}^2_n c_n,\oo{\delta}_n c_n]]
\sim \e Z^*_n
\end{align*}
and, consequently,
\begin{align*}
&\mathbf{E}\left[f\left(\nu_{m(n)},\frac{S_{\nu_{m(n)}}}{B_n},s_n\right);
T_g>n,Z_{\nu_{m(n)}}\in[\oo{\delta}^2_n c_n,\oo{\delta}_n c_n]\right]\\
&\hspace{2cm}\sim \mathbf{E}f(M_{\alpha,\beta})\e Z^*_n\pr(\tau_0>n)
\sim \mathbf{E}f(M_{\alpha,\beta})\pr(T_g>n).
\end{align*}
Plugging this into \eqref{T1.7}, we have
\begin{align*}
\e[f(s_n);T_g>n]\sim \mathbf{E}f(M_{\alpha,\beta})\pr(T_g>n).
\end{align*}
This implies immediately the desired weak convergence. Thus, it remains to show \eqref{T1.8}.

We shall prove \eqref{T1.8} by giving bounds for the expectation on the left hand side in terms
of boundary problems wit constant boundaries. More precisely,
\begin{align*}
&\mathbf{E}\left[f\left(k,\frac{y+g_k}{c_n},s_n\right);y+\min_{j\in[k,n]}(Z_j-Z_k)>0\right]\\
&\hspace{1cm}\le\e \left[f\left(k,\frac{y+g_k}{c_n},s_n\right);\tau_{y+2G_n}>n\right]\\
&\hspace{1cm}=\e \left[f\left(k,\frac{y+g_k}{c_n},s_n\right)|\tau_{y+2G_n}>n\right]\pr(\tau_{y+G_n}>n).
\end{align*}
Note that $|f\left(k,\frac{S_k}{c_n},s_n\right)-f(s_n)|\to0$ uniformly over all
trajectories $s_n$ with $S_k\le G_n+\overline{\delta}_nc_n$. This convergene is also
uniform in $k\le m(n)$. Then, using Lemma~\ref{lem:meander} and \eqref{lem1.eq1}, we get
\begin{align*}
&\mathbf{E}\left[f\left(k,\frac{y+g_k}{c_n},s_n\right);y+\min_{j\in[k,n]}(Z_j-Z_k)>0\right]\\
&\hspace{1cm}\le V(y+2G_n)\pr(\tau_0>n)\e f(M_{\alpha,\beta})(1+o(1)).
\end{align*}
Noting that $V(y+2G_n)\sim V(y)$ for $y\in[\oo{\delta}^2_n c_n,\oo{\delta}_n c_n]$, we obtain the upper bound
\begin{align*}
&\mathbf{E}\left[f\left(k,\frac{y+g_k}{c_n},s_n\right);y+\min_{j\in[k,n]}(Z_j-Z_k)>0\right]\\
&\hspace{1cm}\le V(y)\pr(\tau_0>n)\e f(M_{\alpha,\beta})(1+o(1)).
\end{align*}
By the same argument,
\begin{align*}
&\mathbf{E}\left[f\left(k,\frac{y+g_k}{c_n},s_n\right);y+\min_{j\in[k,n]}(Z_j-Z_k)>0\right]\\
&\hspace{1cm}\ge\e \left[f\left(k,\frac{y+g_k}{c_n},s_n\right);\tau_{y-2G_n}>n\right]\\
&\hspace{1cm}\ge V(y)\pr(\tau_0>n)\e f(M_{\alpha,\beta})(1+o(1)).
\end{align*}
These two estimates imply \eqref{T1.8}. Thus, the proof of the functional limit theorem is completed.

\subsection{Proof of (\ref{U-infinity})}
Since the sequence $\{g_n\}$ is decreasing, the sequence $V(S_n-g_n)\mathbb{I}\{T_g>n\}$ is
a submartingale and, in particular, the sequence $\e[V(S_n-g_n);T_g>n]$ is increasing. Thus,
it suffices to show that $\e[V(S_{2^j}-g_{2^j});T_g>{2^j}]$ converges to $\infty$.
We first note that
\begin{align*}
&\e[V(S_{2^{j+1}}-g_{2^{j+1}});T_g>{2^{j+1}}]\\
&\ge\int_{g_{2^j}}^\infty \pr(S_{2^j}\in dy;T_g>2^j)
\e[V(y+S_{2^j}-g_{2^{j+1}});\tau_{y-g_{2^j}}>2^j]\\
&=\e[V(S_{2^j}-g_{2^j});T_g>2^j]\\
&+\int_{g_{2^j}}^\infty \pr(S_{2^j}\in dy;T_g>2^j)
\e[V(y+S_{2^j}-g_{2^{j+1}})-V(y+S_{2^j}-g_{2^{j}});\tau_{y-g_{2^j}}>2^j],
\end{align*}
where we have used the harmonicity of $V$ in the last step. Furthermore, since all terms
in the integral are positive, we have
\begin{align*}
&\e[V(S_{2^{j+1}}-g_{2^{j+1}});T_g>{2^{j+1}}]-\e[V(S_{2^j}-g_{2^j});T_g>2^j]\\
&\hspace{1cm}\geq\int_{c_{2^j}}^{2c_{2^j}} \pr(S_{2^j}\in dy;T_g>2^j)\\
&\hspace{3cm}\times\e[V(y+S_{2^j}-g_{2^{j+1}})-V(y+S_{2^j}-g_{2^{j}});\tau_{y-g_{2^j}}>2^j].
\end{align*}
Since $V$ is a renewal function, there exists a positive constant $C$ such that
$$
\liminf_{x\to\infty}\frac{x}{V(x)}(V(x+u)-V(x))\ge Cu
$$
for all $u$ large enough.
Therefore,
\begin{align*}
&\e[V(S_{2^{j+1}}-g_{2^{j+1}});T_g>{2^{j+1}}]-\e[V(S_{2^j}-g_{2^j});T_g>2^j]\\
&\hspace{1cm}\geq\int_{c_{2^j}}^{2c_{2^j}} \pr(S_{2^j}\in dy;T_g>2^j)\\
&\hspace{3cm}\times C'(g_{2^j}-g_{2^{j+1}})\frac{V(c_{2^j})}{c_{2^j}}
\pr(S_{2^j}\in[c_{2^j},2c_{2^j}],\tau_{y-g_{2^j}}>2^j).
\end{align*}
Applying now the standard (non-conditional) limit theorem for $S_n$ and Theorem~\ref{T3},
we obtain
\begin{align*}
&\e[V(S_{2^{j+1}}-g_{2^{j+1}});T_g>{2^{j+1}}]-\e[V(S_{2^j}-g_{2^j});T_g>2^j]\\
&\hspace{1cm}\geq C''(g_{2^j}-g_{2^{j+1}})\frac{V(c_{2^j})}{c_{2^j}}\pr(T_g>2^j).
\end{align*}
Combining Theorem~\ref{T1} and \eqref{V-tau}, we have
$$
V(c_{2^j})\pr(T_g>2^j)\sim A \e[V(S_{2^j}-g_{2^j});T_g>2^j].
$$
Consequently,
\begin{align*}
&\e[V(S_{2^{j+1}}-g_{2^{j+1}});T_g>{2^{j+1}}]\ge
\e[V(S_{2^j}-g_{2^j});T_g>2^j]\left(1+C'''\frac{g_{2^j}-g_{2^{j+1}}}{c_{2^j}}\right).
\end{align*}
Iterating this estimate, we obtain
\begin{align*}
&\e[V(S_{2^{j+1}}-g_{2^{j+1}});T_g>{2^{j+1}}]\ge
\e[V(S_{1}-g_{1});T_g>1]\prod_{k=0}^j\left(1+c'''\frac{g_{2^k}-g_{2^{k+1}}}{c_{2^k}}\right).
\end{align*}
It remains to note that the condition $\sum \frac{|g_n|}{nc_n}=\infty$ implies that
the rigth hand side in the previous display goes to infinity as $j\to\infty$.


\printbibliography

\end{document}